\documentclass[reqno, 12
pt]{amsart}
\usepackage{amssymb}
\usepackage{amsmath}
\usepackage[mathscr]{euscript}
\usepackage{mathtools}
\usepackage[top=2cm,bottom=3cm,left=2.5cm,right=2.5cm]{geometry}
\usepackage[T1]{fontenc}% Brzd\k{e}k

\newtheorem{theorem}{Theorem}
\newtheorem{cor}{Corollary}

\newtheorem{definition}{Definition}

\theoremstyle{remark}
\newtheorem{example}{Example}
\newtheorem{remark}{Remark}

\DeclareMathOperator{\Id}{Id}
\DeclareMathOperator{\esssup}{esssup}

\date{November 2024}
\title{$(L^p, L^q)$ Hyers-Ulam stability}

\begin{document}

\author[D. Dragi\v cevi\'c]{Davor Dragi\v cevi\'c}
\address{Faculty of Mathematics, University of Rijeka, Radmile Matej\v ci\' c 2, 51000 Rijeka, Croatia}
\email{ddragicevic@math.uniri.hr}
\author[M. Onitsuka]{Masakazu Onitsuka}
\address{Department of Applied Mathematics, Faculty of Science, Okayama University of Science, Ridai-cho 1-1, Okayama 700-0005, Japan}
\email{onitsuka@ous.ac.jp}

\begin{abstract}
We introduce a new concept of Hyers-Ulam stability, in which in the size of a pseudosolution of a given ordinary differential equation and its deviation from an exact solution are measured with respect to different norms. These norms are associated to $L^p$-spaces for $p\in [1, \infty]$. Our main objective is to formulate sufficient conditions under which semilinear ordinary differential equations  exhibit such property. In addition, in certain special cases we obtain explicit formulas for the best Hyers-Ulam constant.
\end{abstract}
\maketitle

\section{Introduction}
Let $F\colon \mathbb R^d \times [0, \infty)\to \mathbb R^d$ be a continuous function and consider a nonautonomous and nonlinear differential equation given by 
\begin{equation}\label{ODE}
x'=F(t, x), \quad t\ge 0.
\end{equation}
We recall that~\eqref{ODE} is \emph{Hyers-Ulam stable} if there exists $L>0$ such that for each $\varepsilon >0$ and a continuously differentiable function $y\colon [0, \infty)\to \mathbb R^d$ satisfying
\begin{equation}\label{cond1}
    \sup_{t\ge 0}|y'(t)-F(t, y(t))| \le \varepsilon,
\end{equation}
there exists a solution $x\colon [0, \infty) \to \mathbb R^d$ of~\eqref{ODE} such that 
\begin{equation}\label{cond2}
\sup_{t\ge 0}|x(t)-y(t)| \le L\varepsilon.
\end{equation}
Here, $| \cdot |$ is any norm on $\mathbb R^d$. Therefore, the Hyers-Ulam stability corresponds to the requirement that in a vicinity of an approximate solution of~\eqref{ODE} we can find its true solution.

Observe that in this notion,  the ``size'' of the approximate solution $y$ (see~\eqref{cond1}) and its deviation from the true solution $x$ (see~\eqref{cond2}) are both measured with respect to the supremum norm.  Certainly, the supremum norm is just one of the  possible choices of a norm that can be used to measure these two quantities. 
The main objective of the present paper is  to introduce a more general concept of Hyers-Ulam stability in which these two quantities will be measured in a different manner, and not necessarily with respect to the same norm.

More precisely, for $p\in [1, \infty)$, let $L^p([0, \infty), \mathbb R^d)$ consist of all measurable functions $y\colon [0, \infty)\to \mathbb R^d$ such that 
\begin{equation*}
\|y\|_{L^p([0, \infty), \mathbb R^d)}:=\left ( \int_0^\infty |y(t)|^p \, dt\right )^{1/p}<+\infty.
\end{equation*}
Similarly, one can introduce $L^\infty([0, \infty), \mathbb R^d)$.  Take now two indices $p, q\in [1, \infty]$. We say that~\eqref{ODE} exhibits \emph{$(L^p, L^q)$ Hyers-Ulam stability} if there exists $L>0$ with the property that for each $\varepsilon >0$ and a continuously differentiable function $y\colon [0, \infty)\to \mathbb R^d$ satisfying
\begin{equation*}
\|y'(\cdot)-F(\cdot, y(\cdot))\|_{L^q([0, \infty), \mathbb R^d)} \le \varepsilon,
\end{equation*}
there exists a solution $x\colon [0, \infty) \to \mathbb R^d$ of~\eqref{ODE} such that 
\begin{equation}\label{cond3}
\|x-y\|_{L^p([0, \infty), \mathbb R^d)} \le L\varepsilon.
\end{equation}
We note that the classical notion of Hyers-Ulam stability corresponds to the particular case when $p=q=\infty$.

Focusing on the case when~\eqref{ODE} is of semilinear form, i.e. 
\[
F(t, x)=A(t)x+f(t, x),
\]
where $A(t)\colon \mathbb R^d \to \mathbb R^d$ is a linear operator for each $t\ge 0$, our main result (see Theorem~\ref{main:thm}) says that~\eqref{ODE} admits $(L^p, L^q)$ Hyers-Ulam stability for any choice of pairs $(p, q)$ with $1\le q\le p \le \infty$ provided that:
\begin{itemize}
\item $x'=A(t)x$ admits an exponential dichotomy;
\item $f(t, \cdot)$ is Lipschitz for each $t\ge 0$ with a uniform sufficiently small Lipschitz constant $c\ge 0$.
\end{itemize}
This generalizes~\cite[Theorem 2, $H=\mathbb R^n$]{BDOP} (originally discussed in~\cite{BD}),  where the particular case $p=q=\infty$ was considered. Moreover, we obtain an explicit upper bound  for the constant $L$ that appears in~\eqref{cond3},  and illustrate by concrete examples its sharpness.  Finally, in Section~\ref{last} we discuss linear autonomous differential equations on $\mathbb C^2$ and provide some lower bounds on $L$.

\section{Preliminaries}

Let $X=(X, |\cdot|)$ be an arbitrary Banach space. By $\mathcal B(X)$ we will denote the space of all bounded linear operators equipped with the operator norm $\| \cdot \|$. 
Let $A\colon [0, \infty)\to \mathcal B(X)$ and $f\colon [0, \infty)\times X\to X$ be continuous functions. We consider the associated nonlinear and nonautonomous differential equation
\begin{equation}\label{nde}
x'=A(t)x+f(t, x), \quad t\ge 0.
\end{equation}
We will also consider the linear part of~\eqref{nde} given by 
\begin{equation}\label{lde}
x'=A(t)x, \quad t\ge 0.
\end{equation}
Let us recall the notion of an exponential dichotomy. By $T(t,s)$, $t, s\ge 0$ we will denote the evolution family associated to~\eqref{lde}.
\begin{definition}
We say that~\eqref{lde} admits an \emph{exponential dichotomy} if there exist a family of projections $P(t)$, $t\ge 0$ and constants $D, \lambda>0$ such that the following conditions hold:
\begin{enumerate}
\item for $t, s\ge 0$, $T(t,s)P(s)=P(t)T(t,s)$;
\item for $t\ge s\ge 0$,
\begin{equation}\label{ed1}
    \|T(t,s)P(s)\| \le De^{-\lambda (t-s)};
\end{equation}
\item for $0\le t\le s$,
\begin{equation}\label{ed2}
\|T(t,s)(\Id-P(s))\| \le De^{-\lambda (s-t)},
\end{equation}
where $\Id$ denotes the identity operator on $X$.
\end{enumerate}
In addition, if $P(t)=\Id$ for all $t\ge 0$, we say that \eqref{lde} admits an \emph{exponential contraction}; if $P(t)=0$ for all $t\ge 0$, we say that \eqref{lde} admits an \emph{exponential expansion}. 
\end{definition}

It is well-known that the notions of Hyers-Ulam stability and exponential dichotomy are closely related. This relationship has been investigated for both differential equations (see~\cite{BD, BacDra, BDP, BusBarTab,Dra1,ZadShaSha}), as well as  for difference and dynamic equations (see~\cite{BarBusTab,BusO'ReSaiTab,ZadShaIsmLi,ZadZad}) and random dynamical systems~\cite{DZZ}.

For $p\in [1, \infty)$, let $L^p([0, \infty), X)$ consist of all $x\colon [0, \infty) \to X$ which are Bochner measurable and 
\[
\|x\|_{L^p([0, \infty), X)}:=\left (\int_0^\infty |x(t)|^p \, dt \right)^{\frac 1 p}<+\infty.
\]
By identifying functions which are equal almost everywhere, we have that $(L^p, \| \cdot \|_{L^p([0, \infty), X)})$  is a Banach space.

Similarly, let $L^\infty([0, \infty), X)$ consist of all Bochner measurable functions $x\colon [0, \infty)\to X$ such that 
\[
\|x\|_{L^\infty([0, \infty), X)}:=\esssup_{t\ge 0}|x(t)|<+\infty.
\]
Then,  again by identifying functions which are equal almost everywhere, we have that \\ $(L^\infty, \| \cdot \|_{L^\infty([0, \infty), X)})$ is also a Banach space.

\begin{definition}\label{pqhus}
Let $p, q\in [1, \infty]$. We say that~\eqref{nde} admits \emph{$(L^p, L^q)$ Hyers-Ulam stability} if there exists $L>0$ with the property that for each $\varepsilon>0$ and a continuously differentiable  $y\colon [0, \infty) \to X$ satisfying $y'(\cdot)-A(\cdot)y(\cdot)-f(\cdot, y(\cdot))\in L^q([0, \infty), X)$  and 
\begin{equation}\label{y}
    \|y'(\cdot)-A(\cdot)y(\cdot)-f(\cdot, y(\cdot))\|_{L^q([0, \infty), X)}\le \varepsilon,
\end{equation}
there exists a solution $x\colon [0, \infty)\to X$ of~\eqref{nde} such that $x-y\in L^p([0, \infty), X)$ and 
\begin{equation}\label{approx}
\| x-y\|_{L^p([0, \infty), X)}\le L\varepsilon.
\end{equation}
In addition, we say that $L$ is a \emph{Hyers-Ulam stability (HUS) constant}.
\end{definition}
\begin{remark}
We note that the standard concept of Hyers-Ulam stability corresponds to the choice $p=q=\infty$.  In contrast to this case, Definition~\ref{pqhus} allows that the size of an approximate solution $y$ of~\eqref{nde} and its deviation from exact solution $x$ of~\eqref{nde} are measured with respect to different norms.
\end{remark}

\section{Main results}

The following result gives sufficient conditions under which~\eqref{nde} exhibits $(L^p, L^q)$ Hyers-Ulam stability. 

\begin{theorem}\label{main:thm}
Let $\infty \ge p\ge q\ge 1$. Suppose that~\eqref{lde} admits an exponential dichotomy and that there exists $c\in [0, \frac{\lambda}{2D}) $ such that 
\begin{equation}\label{lip}
|f(t, v_1)-f(t, v_2)| \le c|v_1-v_2|, \quad \text{for $v_1, v_2\in X$ and $t\ge 0$.}
\end{equation}
Then, \eqref{nde} admits $(L^p, L^q)$ Hyers-Ulam stability.

\end{theorem}

\begin{proof}
Take $\varepsilon>0$ and let $y\colon [0, \infty) \to X$ be a continuously differentiable function satisfying $y'(\cdot)-A(\cdot)y(\cdot)-f(\cdot, y(\cdot))\in L^q([0, \infty), X)$  and~\eqref{y}.
For $z\in L^p([0, \infty), X)$, we set 
\[
\begin{split}
(\mathcal T_1 z)(t)&:=\int_0^t T(t,s )P(s)(A(s) y(s)+f(s, y(s)+z(s))-y'(s))\, ds,
\end{split}
\]
for $t\ge 0$. Observe that~\eqref{ed1} and~\eqref{lip} imply that 
\[
\begin{split}
&\left |\int_0^t T(t,s )P(s)(A(s) y(s)+f(s, y(s)+z(s))-y'(s))\, ds\right | \\
&\le \int_0^t |T(t,s )P(s)(A(s) y(s)+f(s, y(s)+z(s))-y'(s))| \, ds \\
&\le \int_0^t \|T(t,s) P(s)\| \cdot |A(s) y(s)+f(s, y(s))-y'(s)|\, ds \\
&\phantom{\le}+
\int_0^t \|T(t,s)P(s)\|\cdot  |f(s, y(s)+z(s))-f(s, y(s))|\, ds \\
&\le D\int_0^t e^{-\lambda (t-s)}|A(s) y(s)+f(s, y(s))-y'(s)|\, ds \\
&\phantom{\le}+cD\int_0^t e^{-\lambda (t-s)}|z(s)| \, ds,
\end{split}
\]
for $t\ge 0$.
Let $a_1, a_2\colon \mathbb R \to \mathbb R$ be given by 
\[
a_1(t):=\begin{cases}
\int_0^t e^{-\lambda (t-s)}|A(s) y(s)+f(s, y(s))-y'(s)|\, ds & t\ge 0; \\
0 & t<0,
\end{cases}
\]
and 
\[
a_2(t):=\begin{cases}
    \int_0^t e^{-\lambda (t-s)}|z(s)| \, ds & t\ge 0; \\
0 &t<0.
\end{cases}
\]
 Observe that $a_i=b\ast c_i$ for $i\in \{1, 2\}$, where $\ast$ denotes the convolution, 
\[
b(t)=\begin{cases}
e^{-\lambda t} &t\ge 0; \\
0& t<0,
\end{cases}
\]
\[
c_1(t)=\begin{cases}
|A(t)y(t)+f(t, y(t))-y'(t)| &t\ge 0; \\
0 &t<0,
\end{cases} \quad \text{and} \quad
c_2(t)=\begin{cases}
|z(t)| & t\ge 0;\\
0& t<0.
\end{cases}
\]
Choose $r\in [1, \infty]$ such that $\frac 1 p+1=\frac 1 q+\frac 1 r$. By applying Young's inequality, we obtain that 
\[
\|a_1\|_{L^p(\mathbb R)}\le \|b\|_{L^r(\mathbb R)} \cdot \|c_1\|_{L^q(\mathbb R)} \le \left (\frac{1}{\lambda r}\right )^{\frac 1 r}\|c_1\|_{L^q(\mathbb R)},
\]
and similarly
\[
\|a_2\|_{L^p(\mathbb R)}\le \|b\|_{L^1(\mathbb R)} \cdot \|c_2\|_{L^p(\mathbb R)}\le \frac{1}{\lambda}\|c_2\|_{L^p(\mathbb R)}.
\]
Here (and in the rest of the paper) we use convention  $0^0 :=1$ when $r=\infty$.
We conclude (see also~\eqref{y}) that $\mathcal T_1 z\in L^p([0, \infty), X)$ and 
\begin{equation}\label{z1}
\|\mathcal T_1 z\|_{L^p([0, \infty), X)}\le D\left (\frac{1}{\lambda r}\right )^{\frac 1 r}\varepsilon+\frac{cD}{\lambda}\|z\|_{L^p([0, \infty), X)}.
\end{equation}

We now set 
\[
(\mathcal T_2 z)(t):=-\int_t^\infty T(t,s)(\Id-P(s))(A(s) y(s)+f(s, y(s)+z(s))-y'(s))\, ds,
\]
for $t\ge 0$ and $z\in L^p([0, \infty), X)$. By~\eqref{ed2} and~\eqref{lip}, we have that 
\[
\begin{split}
&\left |\int_t^\infty T(t,s )(\Id-P(s))(A(s) y(s)+f(s, y(s)+z(s))-y'(s))\, ds\right | \\
&\le \int_t^\infty \|T(t,s)(\Id- P(s))\| \cdot |A(s) y(s)+f(s, y(s))-y'(s)|\, ds \\
&\phantom{\le}+
\int_t^\infty \|T(t,s)(\Id-P(s))\|\cdot  |f(s, y(s)+z(s))-f(s, y(s))|\, ds \\
&\le D\int_t^\infty e^{-\lambda (s-t)}|A(s) y(s)+f(s, y(s))-y'(s)|\, ds \\
&\phantom{\le}+cD\int_t^\infty e^{-\lambda (s-t)}|z(s)| \, ds,
\end{split}
\]
for $t\ge 0$. Let $\bar a_1, \bar a_2\colon \mathbb R \to \mathbb R$ be given by 
\[
\bar a_1(t):=\begin{cases}
\int_t^\infty e^{-\lambda (s-t)}|A(s) y(s)+f(s, y(s))-y'(s)|\, ds & t\ge 0; \\
0 & t<0,
\end{cases}
\]
and 
\[
\bar a_2(t):=\begin{cases}
    \int_t^\infty e^{-\lambda (s-t)}|z(s)| \, ds & t\ge 0; \\
0 &t<0.
\end{cases}
\]
 Then,  $\bar a_i=\bar b\ast c_i$ for $i\in \{1, 2\}$, where 
\[
\bar b(t)=\begin{cases}
0 &t\ge 0; \\
e^{\lambda t}& t<0.
\end{cases}
\]
By Young's inequality, we have that 
\[
\|\bar a_1\|_{L^p(\mathbb R)}\le \|\bar b\|_{L^r(\mathbb R)} \cdot \|c_1\|_{L^q(\mathbb R)} \le \left (\frac{1}{\lambda r}\right )^{\frac 1 r}\|c_1\|_{L^q(\mathbb R)}
\]
and 
\[
\|\bar a_2\|_{L^p(\mathbb R)}\le \frac{1}{\lambda}\|c_2\|_{L^p(\mathbb R)}.
\]
Hence, $\mathcal T_2 z\in L^p([0, \infty), X)$ and 
\begin{equation}\label{z2}
\|\mathcal T_2 z\|_{L^p([0, \infty), X)}\le D\left (\frac{1}{\lambda r}\right )^{\frac 1 r}\varepsilon+\frac{cD}{\lambda}\|z\|_{L^p([0, \infty), X)}.
\end{equation}
Let 
\[
\mathcal Tz:=\mathcal T_1 z+\mathcal T_2z, \quad z\in L^p([0, \infty), X).
\]
By~\eqref{z1} and~\eqref{z2} we have that $\mathcal Tz\in L^p([0, \infty), X)$ and 
\begin{equation}\label{Tz}
    \|\mathcal Tz\|_{L^p([0, \infty), X)} \le 2D\left (\frac{1}{\lambda r}\right )^{\frac 1 r}\varepsilon+\frac{2cD}{\lambda}\|z\|_{L^p([0, \infty), X)}, \quad z\in L^p([0, \infty), X).
\end{equation}
In particular, we have that 
\begin{equation}\label{zero}
\|\mathcal T0\|_{L^p([0, \infty), X)} \le 2D\left (\frac{1}{\lambda r}\right )^{\frac 1 r}\varepsilon.
\end{equation}

Take now $z_i\in L^p([0, \infty), X)$, $i\in \{1, 2\}$. 
We have that 
\[
(\mathcal T_1 z_1)(t)-(\mathcal T_1 z_2)(t)=\int_0^t T(t,s)P(s)(f(s, y(s)+z_1(s))-f(s, y(s)+z_2(s)))\, ds, \quad t\ge 0.
\]
By~\eqref{ed1} and~\eqref{lip}, we have that 
\[
|(\mathcal T_1 z_1)(t)-(\mathcal T_1 z_2)(t)|\le cD\int_0^te^{-\lambda (t-s)}|z_1(s)-z_2(s)|\, ds, \quad t\ge 0.
\]
Using Young's inequality gives that 
\begin{equation}\label{lip1}
\|\mathcal T_1 z_1-\mathcal T_1 z_2\|_{L^p([0, \infty), X)}\le \frac{cD}{\lambda}\|z_1-z_2\|_{L^p([0, \infty), X)}.
\end{equation}
Similarly, 
\begin{equation}\label{lip2}
\|\mathcal T_2 z_1-\mathcal T_2 z_2\|_{L^p([0, \infty), X)}\le \frac{cD}{\lambda}\|z_1-z_2\|_{L^p([0, \infty), X)}.
\end{equation}
From~\eqref{lip1} and~\eqref{lip2} it follows that 
\begin{equation}\label{contr}
\|\mathcal T z_1-\mathcal T z_2\|_{L^p([0, \infty), X)}\le \frac{2cD}{\lambda}\|z_1-z_2\|_{L^p([0, \infty), X)}, \quad \text{for $z_1, z_2\in L^p([0, \infty), X)$.}
\end{equation}
Let
\[
L:=\frac{2D\left (\frac{1}{\lambda r}\right )^{\frac 1 r}}{1-\frac{2cD}{\lambda}}>0,
\]
and define
\[
\mathcal D:=\left \{z\in L^p([0, \infty), X): \ \|z\|_{L^p([0, \infty), X)}\le L\varepsilon \right \}.
\]
By~\eqref{zero} and~\eqref{contr}, we have that 
\[
\begin{split}
\|\mathcal T z\|_{L^p([0, \infty), X)}&\le \|\mathcal T 0\|_{L^p([0, \infty), X)}+\|\mathcal T z-\mathcal T 0\|_{L^p([0, \infty), X)} \\
&\le 2D\left (\frac{1}{\lambda r}\right )^{\frac 1 r}\varepsilon+\frac{2cD}{\lambda}\|z\|_{L^p([0, \infty), X)} \\
&\le 2D\left (\frac{1}{\lambda r}\right )^{\frac 1 r}\varepsilon+
\frac{2cD}{\lambda}L\varepsilon \\
&=L\varepsilon.
\end{split}
\]
Hence, $\mathcal T\mathcal D\subset \mathcal D$ and thus due to~\eqref{contr} we have that $\mathcal T$ is a contraction on $\mathcal D$. Consequently, $\mathcal T$ has a unique fixed point $z\in \mathcal D$. Then, $x:=y+z$ is a solution of~\eqref{nde}, $x-y=z\in L^p([0, \infty), X)$ and 
\[
\|x-y\|_{L^p([0, \infty), X)}=\|z\|_{L^p([0, \infty), X)}\le L\varepsilon.
\]
We conclude that~\eqref{approx} holds.
The proof of the theorem is completed. 
\end{proof}

\begin{cor}\label{Lpqlinear}
Let $\infty \ge p\ge q\ge 1$ and suppose that~\eqref{lde} admits an exponential dichotomy. Then, \eqref{lde} admits $(L^p, L^q)$ Hyers-Ulam stability.
\end{cor}

\begin{proof}
We apply the previous theorem for $f\equiv 0$.
\end{proof}

\begin{remark}
Suppose that~\eqref{lde} admits an exponential dichotomy and  choose an arbitrary continuous function $w\colon [0, \infty)\to \mathbb R$ which is unbounded (i.e. $\sup_{t\ge 0}|w(t)|=+\infty$) but belongs to $L^q([0, \infty), \mathbb R)$ for some $q\in [1, \infty)$. Let $y$ be a solution of the inhomogeneous equation
\[
y'=A(t)y+w(t).
\]
Then, $y$ cannot be regarded as a pseudo-solution of~\eqref{lde} in the classical sense (see~\eqref{cond1} for $F(t,x)=A(t)x$) as $\sup_{t\ge 0}|y'(t)-A(t)y(t)|=+\infty$. However, it does satisfy~\eqref{y} for $f\equiv 0$. This shows that Theorem~\ref{main:thm} is applicable beyond its previously known version when $p=q=\infty$.
\end{remark}

\begin{remark}\label{rem:HUSconst}
%The constant $L$ that appears in the previous theorem is known as the Hyers-Ulam stability constant (HUS constant or Ulam constant). This is a value that represents the upper bound of the error between the approximate solution $y$ and the true solution $x$, and plays a major role in applications (see, \cite{KonOni,Oni1,Oni2}). 
From the proof of Theorem~\ref{main:thm}, we see that the HUS constant in Theorem~\ref{main:thm} is given by
\[
L=\frac{2D\left (\frac{1}{\lambda r}\right )^{\frac 1 r}}{1-\frac{2cD}{\lambda}},
\]
where $r$ satisfies $\frac 1 p+1=\frac 1 q+\frac 1 r$. 
Moreover, we see that the HUS constant in Corollary~\ref{Lpqlinear} is
\[
L=2D\left (\frac{1}{\lambda r}\right )^{\frac 1 r}.
\]

In special cases the values of these constants can be nearly halved. If we assume that \eqref{lde} admits an exponential contraction or an exponential expansion on $[0,\infty)$, then the HUS constant in Theorem~\ref{main:thm} is 
\begin{equation}\label{NL}L=\frac{D\left (\frac{1}{\lambda r}\right )^{\frac 1 r}}{1-\frac{cD}{\lambda}}\end{equation} 
with $c\in [0, \frac{\lambda}{D})$. Also, under the same assumption, we have the HUS constant in Corollary~\ref{Lpqlinear} is
\[
L=D\left (\frac{1}{\lambda r}\right )^{\frac 1 r}.
\]
To briefly outline the proof, if we assume that \eqref{lde} admits an exponential contraction or an exponential expansion on $[0,\infty)$ in the proof of Theorem~\ref{main:thm}, one of the inequalities  \eqref{z1} or \eqref{z2} is redundant. Therefore, the inequality given by \eqref{Tz} can be rewritten as: 
\begin{equation*}
    \|\mathcal Tz\|_{L^p([0, \infty), X)} \le D\left (\frac{1}{\lambda r}\right )^{\frac 1 r}\varepsilon+\frac{cD}{\lambda}\|z\|_{L^p([0, \infty), X)}, \quad z\in L^p([0, \infty), X).
\end{equation*}
Using this with $c\in [0, \frac{\lambda}{D})$, we can conclude that the HUS constant is given by~\eqref{NL}. 
\end{remark}

\begin{remark}
Assume that~\eqref{lde} admits an exponential expansion and take $\infty \ge p\ge q\ge 1$. Then, for each $\varepsilon>0$ and a continuously differentiable $y\colon [0, \infty) \to X$ such that $\|y'(\cdot)-A(\cdot)y(\cdot)\|_{L^q([0, \infty), X)}\le \varepsilon$, there exists a \emph{unique} solution $x\colon [0, \infty) \to X$ of~\eqref{lde} such that $\|x-y\|_{L^p([0, \infty), X)}\le L\varepsilon$, where $L$ is the HUS constant. In fact, $x$ is the unique solution of~\eqref{lde} with the property that $\|x-y\|_{L^p([0, \infty), X)}<+\infty$.

In order to show this, suppose that $x, \tilde x\colon [0, \infty)\to X$ are two solutions of~\eqref{lde} such that 
\[
\|x-y\|_{L^p([0, \infty), X)}<+\infty \quad \text{and} \quad \|\tilde x-y\|_{L^p([0, \infty), X)}<+\infty.
\]
Then, $z:=x-\tilde x$ is a solution of~\eqref{lde} such that $\|z\|_{L^p([0, \infty), X)}<+\infty$. On the other hand, 
\[
|z(0)|=|T(0, t)z(t)| \le De^{-\lambda t}|z(t)|,
\]
and thus 
\[
\frac 1 D e^{\lambda t} |z(0)| \le |z(t)|, \quad t\ge 0.
\]
Since $\|z\|_{L^p([0, \infty), X)}<+\infty$, we find that $z(0)=0$ and thus $z\equiv 0$. We conclude that $x=\tilde x$.
\end{remark}

We now present several examples  to illustrate our results.

\begin{example}\label{ex1}
We consider the nonlinear differential equation
\begin{equation}\label{ex:sin}
  x' = ax+b\sin x, \quad t\ge 0,
\end{equation}
where $a>0$, $b\in \mathbb{R}$ and $|b| \in [0, a)$. This equation is a special case of \eqref{nde} when $n=1$, $A(t)=a$ and $f(t,x)=b\sin x$. The linear part of~\eqref{ex:sin} given by 
\begin{equation}\label{ex:lde}
x'=ax, \quad t\ge 0.
\end{equation}
Then we see that $x(t)=Ce^{at}$ is the general solution of \eqref{ex:lde}, where $C\in \mathbb{R}$ is an arbitrary constant. Thus, we obtain the associated  evolution family  is given by $T(t,s)=e^{a(t-s)}$ for $t, s\ge 0$. Consequently,
\[ \|T(t,s)\| \le e^{-a(s-t)} \]
for $0\le t\le s$. That is, \eqref{ed2} holds with $D=1$, $\lambda = a$ and $P(t)=0$ for $t\ge 0$. Hence, \eqref{ex:lde} admits an exponential dichotomy. More precisely, \eqref{ex:lde} admits an exponential expansion. From
\[ \left|\frac{b\sin v_1 - b\sin v_2}{v_1-v_2}\right| = |b|\left|\frac{2\sin \frac{v_1 + v_2}{2} \sin\frac{v_1 - v_2}{2}}{v_1-v_2}\right| \le |b|, \]
we have $|b\sin v_1 - b\sin v_2| \le |b||v_1-v_2|$ for $v_1, v_2\in \mathbb{R}$. That is, \eqref{lip} holds with $c=|b|$. Note that the assumption $|b| \in [0, a)$ implies that $c\in [0, \frac{\lambda}{D})$. 

Let $\infty \ge p\ge q\ge 1$. By Theorem~\ref{main:thm}, we conclude that there exists $L>0$ such that for each $\varepsilon>0$ and continuously differentiable function $y\colon [0, \infty) \to \mathbb{R}$ such that $y'(\cdot)-ay(\cdot)-b\sin y(\cdot)\in L^q([0, \infty), \mathbb{R})$
and 
\[ \|y'(\cdot)-ay(\cdot)-b\sin y(\cdot)\|_{L^q([0, \infty), \mathbb{R})}\le \varepsilon, \]
there exists a solution $x\colon [0, \infty)\to \mathbb{R}$ of~\eqref{ex:sin} such that $x-y\in L^p([0, \infty), \mathbb{R})$ and 
\[ \| x-y\|_{L^p([0, \infty), \mathbb{R})}\le L\varepsilon. \]
In addition, from Remark~\ref{rem:HUSconst} we know that the HUS constant $L$ is given by
\[
L=\frac{\left (\frac{1}{a r}\right )^{\frac 1 r}}{1-\frac{|b|}{a}}, 
\]
where 
\begin{equation}\label{r}
r=\begin{cases}
\frac{pq}{q-p+pq} & \text{if $p<\infty$;} \\
\infty & \text{if $p=\infty$ and $q=1$;} \\
\frac{q}{q-1} & \text{if $p=\infty$ and $q\in (1, \infty)$;} \\
1 & \text{if $p=q=\infty$.}
\end{cases}
\end{equation}
\end{example}

In the next example, we consider an equation with a more specific perturbation term and illustrate the sharpness of our results. More precisely, we explain that the obtained HUS constant is sharp. 

\begin{example}\label{ex2}
Consider the equation
\begin{equation}\label{ex2:pertsin}
  y' = ay+b\sin y+e^{-\gamma t}, \quad t\ge 0,
\end{equation}
where $a>0$, $b\in \mathbb{R}$, $|b| \in [0, a)$ and $\gamma > 0$. This equation can be seen as a perturbation of~\eqref{ex:sin}. Let $y\colon [0, \infty)\to \mathbb R$ be a solution of~\eqref{ex2:pertsin}. Then,
\[ \|y' - ay-b\sin y\|_{L^q([0, \infty), \mathbb{R})} = \|e^{-\gamma t}\|_{L^q([0, \infty), \mathbb{R})} = \left(\frac{1}{q\gamma}\right)^{\frac{1}{q}} < \infty. \]
Using the facts given in Example~\ref{ex1} with $\varepsilon = \left(\frac{1}{q\gamma}\right)^{\frac{1}{q}}$, we have that there exists a solution $x\colon [0, \infty)\to \mathbb{R}$ of~\eqref{ex:sin} such that $x-y\in L^p([0, \infty), \mathbb{R})$ and 
\[ \| x-y\|_{L^p([0, \infty), \mathbb{R})}\le \frac{\left (\frac{1}{a r}\right )^{\frac 1 r}}{1-\frac{|b|}{a}}\left(\frac{1}{q\gamma}\right)^{\frac{1}{q}}, \]
where $r$ is given by~\eqref{r}. 

For the sake of simplicity, we consider the case that $b=0$ and $p=q$ below. That is, we consider equations~\eqref{ex:lde} and \eqref{ex2:pertsin} with $b=0$. 
From the above facts, we obtain HUS constant $\frac{1}{a}$ and $\varepsilon = \left(\frac{1}{p\gamma}\right)^{\frac{1}{p}}$. 
In this case, we have a solution
\[ y(t) = -\frac{1}{a+\gamma}e^{-\gamma t} \]
of \eqref{ex2:pertsin}. Since $\|y' - ay\|_{L^p([0, \infty), \mathbb{R})} = \left(\frac{1}{p\gamma}\right)^{\frac{1}{p}} < \infty$, we see that there exists a solution $x\colon [0, \infty)\to \mathbb{R}$ of~\eqref{ex:lde} such that $x-y\in L^p([0, \infty), \mathbb{R})$ and 
\[ \| x-y\|_{L^p([0, \infty), \mathbb{R})}\le \frac{1}{a}\left(\frac{1}{p\gamma}\right)^{\frac{1}{p}}. \]
Let us identify $x$. Consider the general solution $\xi(t) = C e^{at}$ of \eqref{ex:lde}, where $C\in \mathbb{R}$ is an arbitrary constant.  If we assume that $C\ne0$, then 
\[ \| \xi-y\|_{L^p([0, \infty), \mathbb{R})} = \left\|C e^{at}+\frac{1}{a+\gamma}e^{-\gamma t}\right\|_{L^p([0, \infty), \mathbb{R})} = \left (\int_0^\infty \left|C e^{at}+\frac{1}{a+\gamma}e^{-\gamma t}\right|^p \, dt \right)^{\frac 1 p} = \infty. \]
Therefore, $x\equiv 0$ and consequently 
\[ \| x-y\|_{L^p([0, \infty), \mathbb{R})} = \left\|\frac{1}{a+\gamma}e^{-\gamma t}\right\|_{L^p([0, \infty), \mathbb{R})} = \frac{1}{a+\gamma}\left(\frac{1}{p\gamma}\right)^{\frac{1}{p}} < \infty. \]
Since $\varepsilon=\left(\frac{1}{p\gamma}\right)^{\frac{1}{p}}$ and $\gamma>0$ was arbitrary,  we can conclude that the HUS constant is greater or equal than $\frac{1}{a}$.  
Since the HUS constant obtained by our theorem in this case is $\frac{1}{a}$, we observe  that $\frac 1 a$ is the smallest HUS constant. This illustrates the sharpness of our results. 
\end{example}

\begin{example}
The purpose of this example is to illustrate that the condition that $p\ge q$ cannot be omitted in the statement of Theorem~\ref{main:thm}. To this end, we essentially rely on~\cite[Example 3.15]{AS}. 

Take $1\le p<q \le +\infty$, $\delta \in (p, q)$, $X=\mathbb R$ and consider~\eqref{lde} with $A(t)=-1$ for $t\ge 0$. Moreover, let $y\colon [0, \infty)\to \mathbb R$ be any solution of the equation
\[
y'(t)=-y(t)+\frac{1}{(1+t)^{1/\delta}}, \quad t\ge 0.
\]
Then, $\|y'+y\|_{L^q([0, \infty), \mathbb R)}<+\infty$. On the other hand, there does not exist a solution $x\colon [0, \infty)\to \mathbb R$ of~\eqref{lde} such that $\|x-y\|_{L^p([0, \infty), \mathbb R)}<+\infty$. Indeed, assuming that such $x$ exists, we have that $z:=y-x\in L^p([0, \infty), \mathbb  R)$ satisfies that 
\[
z'(t)=-z(t)+\frac{1}{(1+t)^{1/\delta}}, \quad t\ge 0.
\]
By the variation of constants formula, it yields that 
\[
z(t)=e^{-t}z(0)+\int_0^t e^{-(t-s)}\frac{1}{(1+s)^{1/\delta}}\, ds, \quad t\ge 0.
\]
Then, the same argument as in~\cite[Example 3.15]{AS} yields a contradiction.
\end{example}

We also establish a converse type result showing that the assumption that~\eqref{lde} admits an exponential dichotomy in the statements of Theorem~\ref{main:thm} and Corollary~\ref{Lpqlinear} is natural.

\begin{theorem}\label{conv}
Assume the following:
\begin{enumerate}
\item[(i)] \eqref{lde} admits $(L^p, L^q)$ Hyers-Ulam stability for $p, q\in [1, \infty)$;
\item[(ii)] \eqref{lde} has bounded growth, i.e. there exist $C, a>0$ such that 
\[
\|T(t,s)\| \le Ce^{a(t-s)} \quad t\ge s\ge 0;
\]
\item[(iii)] 
\[
\left \{ v\in X: \ t\mapsto T(t, 0)v\in L^p([0, \infty), X)\right \}
\]
is a closed and complemented subspace of $X$.
\end{enumerate}
Then, \eqref{lde} has exponential dichotomy.
\end{theorem}

\begin{proof}
Take $w\colon [0, \infty)\to X$ continuous such that $w(0)=0$, $\lim\limits_{t\to \infty}w(t)=0$ and $w\in L^q([0, \infty), X)$. Let $y\colon [0, \infty)\to X$ be a solution of the equation
\[
y'(t)=A(t)y(t)+w(t), \quad t\ge 0.
\]
Then, $\|y'(\cdot)-A(\cdot)y(\cdot)\|_{L^q([0, \infty), X)}=\|w\|_{L^q([0, \infty), X)}<+\infty$. Since~\eqref{lde}
admits $(L^p, L^q)$ Hyers-Ulam stability, there exists a solution $x\colon [0, \infty)\to X$ of~\eqref{lde} such that 
\[\|x-y\|_{L^p([0, \infty), X)}\le L\|w\|_{L^q([0, \infty), X)}<+\infty,
\]
where $L$ is the associated HUS constant. Set $z:=y-x$. Then, $z\in L^p([0, \infty), X)$ and 
\[
z'(t)=A(t)z(t)+w(t), \quad t\ge 0.
\]
Consequently, 
\[
z(t)=T(t, s)z(s)+\int_s^tT(t, \tau)w(\tau)\, d\tau, \quad t\ge s\ge 0.
\]
By~\cite[Theorem 3.3.(i)]{SS} we conclude that~\eqref{lde} admits an exponential dichotomy.
\end{proof}

\begin{remark}
We note that in the statement of Theorem~\ref{conv} we haven't required that $p\ge q$.
\end{remark}

\section{Lower bound of HUS constants for 2-dimensional autonomous linear systems}\label{last}

In this section, we consider the autonomous linear system
\begin{equation}\label{2Dlde}
x'=Ax, \quad t\ge 0,
\end{equation}
where $A \in \mathbb{C}^{2\times2}$ is a $2\times2$ constant matrix. 
Throughout  this section, we will use the maximum norm for the vector $v$ defined by 
\[ |v|_{\infty} = \left|\begin{pmatrix}
v_{1}\\
v_{2}
\end{pmatrix}\right|_{\infty} := \max\{|v_{1}|, |v_{2}|\}, \]
with the operator norm for the matrix $\Phi$ defined by
\[ \|\Phi\|_{\infty} = \left\|\begin{pmatrix}
\phi_{11} & \phi_{12}\\
\phi_{21} & \phi_{22}
\end{pmatrix}\right\|_{\infty} := \max\{|\phi_{11}|+|\phi_{12}|, |\phi_{21}|+|\phi_{22}|\}. \]
Let $M$ be a $2\times2$ invertible constant matrix. Using $z= M^{-1}x$, \eqref{2Dlde} is transformed into the equation
\begin{equation}\label{2Dlde2}
z'= M^{-1}AMz, \quad t\ge 0.
\end{equation}
It is well known that if $A$ is invertible and $M$ is chosen appropriately, $M^{-1}AM$ admits one of the following two forms:
\[ \text{(i)}\;\;\begin{pmatrix}
\mu_1 & 0\\
0 & \mu_2
\end{pmatrix}, \quad\text{or}\quad \text{(ii)}\;\;\begin{pmatrix}
\nu & 1\\
0 & \nu
\end{pmatrix}, \]
where $\mu_1$, $\mu_2$, $\nu\in \mathbb{C}$. 
In addition, we can find the fundamental matrices for \eqref{2Dlde2} as
\[ \text{(i)}\;\;\Phi_1(t) := \begin{pmatrix}
e^{\mu_1 t} & 0\\
0 & e^{\mu_2 t}
\end{pmatrix}, \quad \text{(ii)}\;\; \Phi_2(t) := e^{\nu t}\begin{pmatrix}
1 & t\\
0 & 1
\end{pmatrix} \]
for these two cases. Also, we see that for $j \in \{1, 2\}$ and $t$, $s \ge 0$, \[T_j(t,s)=M\Phi_j(t)(M\Phi_j(s))^{-1} \] is the evolution family associated to~\eqref{2Dlde}. Thus,
\begin{align*}
 \|T_j(t,s)\| &= \|M\Phi_j(t)\Phi_j^{-1}(s)M^{-1}\| \le \|M\|\|M^{-1}\|\|\Phi_j(t)\Phi_j^{-1}(s)\|, \quad \text{for $t, s\ge 0$.}
\end{align*}
Now we assume that \eqref{2Dlde} admits an exponential expansion on $[0,\infty)$. 
Then we have
\begin{align*}
 \|T_1(t,s)\|_{\infty} & \le \|M\|_{\infty}\|M^{-1}\|_{\infty}\left\|\begin{pmatrix}
e^{-\mu_1(s-t)} & 0\\
0 & e^{-\mu_2(s-t)}
\end{pmatrix}\right\|_{\infty}\\
  &= \|M\|_{\infty}\|M^{-1}\|_{\infty}\max\{|e^{-\mu_1(s-t)}|, |e^{-\mu_2(s-t)}|\}
\end{align*}
and
\begin{align*}
 \|T_2(t,s)\|_{\infty} & \le \|M\|_{\infty}\|M^{-1}\|_{\infty}\left\| e^{-\nu(s-t)}\begin{pmatrix}
1 & -(s-t)\\
0 & 1
\end{pmatrix}\right\|_{\infty} \\
  &= \|M\|_{\infty}\|M^{-1}\|_{\infty}(1+s-t) |e^{-\nu(s-t)}|
\end{align*}
for $0\le t\le s$. Since \eqref{2Dlde} admits an exponential expansion on $[0,\infty)$, we can conclude that $\Re(\mu_1)>0$, $\Re(\mu_2)>0$, $\Re(\nu)>0$ hold, and that
\begin{align*}
 \|T_1(t,s)\|_{\infty} & \le \|M\|_{\infty}\|M^{-1}\|_{\infty}e^{-\min\{\Re(\mu_1),\Re(\mu_2)\}(s-t)}
\end{align*}
and
\begin{align*}
 \|T_2(t,s)\|_{\infty} & \le \|M\|_{\infty}\|M^{-1}\|_{\infty}(1+s-t)e^{-\Re(\nu)(s-t)} \le \|M\|_{\infty}\|M^{-1}\|_{\infty}\frac{e^{\delta-1}}{\delta}e^{-(\Re(\nu)-\delta)(s-t)}
\end{align*}
for $0\le t\le s$, and for any $0<\delta< \min\{1,\Re(\nu)\}$, where $\Re(w)$ denotes the real part of $w \in \mathbb{C}$. Here we used that $(1+\tau)e^{-\delta \tau}\le \frac{e^{\delta-1}}{\delta}$ for $\tau\ge 0$.
From these facts and Remark~\ref{rem:HUSconst}, we obtain the following result. 

\begin{cor}\label{Lpqlinear2}
Let $\infty \ge p\ge q\ge 1$. Suppose that $A$ is invertible and \eqref{2Dlde} admits an exponential expansion on $[0,\infty)$. Then there exists $L>0$ such that for each $\varepsilon>0$ and  differentiable function $y\colon [0, \infty) \to \mathbb{C}^2$ such that $y'(\cdot)-Ay(\cdot)\in L^q([0, \infty), \mathbb{C}^2)$  and
\begin{equation}\label{2Dy}
    \|y'(\cdot)-Ay(\cdot)\|_{L^q([0, \infty), \mathbb{C}^2)}\le \varepsilon,
\end{equation}
there exists a solution $x\colon [0, \infty)\to \mathbb{C}^2$ of~\eqref{2Dlde} satisfying $x-y\in L^p([0, \infty), \mathbb{C}^2)$ and 
\begin{equation}\label{2Dapprox}
\| x-y\|_{L^p([0, \infty), \mathbb{C}^2)}\le L\varepsilon.
\end{equation}
In addition, the HUS constant $L$ is given by the following:
\begin{itemize}
\item[(i)] If  $M$ is such that  $M^{-1}AM=\begin{pmatrix}
\mu_1 & 0\\
0 & \mu_2
\end{pmatrix}$, then
\[ L=\|M\|_{\infty}\|M^{-1}\|_{\infty}\left(\frac{1}{r\min\{\Re(\mu_1),\Re(\mu_2)\}}\right )^{\frac 1 r}, \]
where $\mu_1$, $\mu_2\in \mathbb{C}$ are the eigenvalues of $A$ with $\Re(\mu_1)>0$ and $\Re(\mu_2)>0$, and $r$ is given by~\eqref{r}. 
\item[(ii)] If  $M$ is such that $M^{-1}AM=\begin{pmatrix}
\nu & 1\\
0 & \nu
\end{pmatrix}$, then, for any $0<\delta< \min\{1,\Re(\nu)\}$, 
\[ L=\|M\|_{\infty}\|M^{-1}\|_{\infty}\frac{e^{\delta-1}}{\delta}\left(\frac{1}{r(\Re(\nu)-\delta)}\right )^{\frac 1 r}, \]
where $\nu\in \mathbb{C}$ is the unique eigenvalue of $A$ with $\Re(\nu)>0$, and $r$ is given by~\eqref{r}.
\end{itemize}
\end{cor}

If \eqref{2Dlde} admits an exponential expansion on $[0,\infty)$, then we can find the lower bound of the HUS constant $L$ in Corollary~\ref{Lpqlinear2}. 

\begin{theorem}\label{Lower:thm}
Let $\infty \ge p\ge q\ge 1$. Suppose that $A$ is invertible and that~\eqref{2Dlde} admits an exponential expansion on $[0,\infty)$. 
%Suppose also that there exists $L>0$ such that for each $\varepsilon>0$ and  differentiable function $y\colon [0, \infty) \to \mathbb{C}^2$ such that $y'(\cdot)-Ay(\cdot)\in L^q([0, \infty), \mathbb{C}^2)$  and \eqref{2Dy}, there exists a solution $x\colon [0, \infty)\to \mathbb{C}^2$ of~\eqref{2Dlde} such that $x-y\in L^p([0, \infty), \mathbb{C}^2)$ and \eqref{2Dapprox} hold.  
Then any HUS constant $L$ satisfies
\[ L \ge \begin{cases}
\frac{|A^{-1}u|(q\gamma)^{\frac{1}{q}}}{|(\gamma A^{-1}+I_2)u|(p\gamma)^{\frac{1}{p}}} & \text{if $p<\infty$;} \\
\frac{|A^{-1}u|(q\gamma)^{\frac{1}{q}}}{|(\gamma A^{-1}+I_2)u|} & \text{if $p=\infty$ and $q\in [1, \infty)$;} \\
\frac{|A^{-1}u|}{|(\gamma A^{-1}+I_2)u|} & \text{if $p=q=\infty$,}
\end{cases} \]
where $I_2$ is the $2\times 2$ identity matrix, $u$ is an arbitrary unit vector and $\gamma$ is an arbitrary positive constant. 
\end{theorem}

\begin{proof}
In the proof below, we prove it for the case $p<\infty$, but if $q=\infty$ or $p=\infty$, we can consider $(q\gamma)^{\frac{1}{q}}$ or $(p\gamma)^{\frac{1}{p}}$ appearing below as 1, respectively. 

Let $\varepsilon$ and $\gamma$ be arbitrary positive constants, and $u$ be an arbitrary unit vector. Set
\[ y(t) = -\frac{(q\gamma)^{\frac{1}{q}}\varepsilon e^{-\gamma t}}{|(\gamma A^{-1}+I_2)u|}A^{-1}u, \quad t \ge 0. \]
Then we have
\[ y'(t) - Ay(t) = \frac{(q\gamma)^{\frac{1}{q}}\varepsilon e^{-\gamma t}}{|(\gamma A^{-1}+I_2)u|}(\gamma A^{-1}+I_2)u, \]
and thus,
\[ |y'(t) - Ay(t)| = (q\gamma)^{\frac{1}{q}}\varepsilon e^{-\gamma t} \]
for $t \ge 0$. This implies that
\[ \|y'(\cdot) - Ay(\cdot)\|_{L^q([0, \infty), \mathbb{C}^2)} = \varepsilon < \infty. \]
Hence $y$ satisfies \eqref{2Dy}. Since $L$ is the HUS constant, there exists a solution $x\colon [0, \infty)\to \mathbb{C}^2$ of~\eqref{2Dlde} such that $x-y\in L^p([0, \infty), \mathbb{C}^2)$ and \eqref{2Dapprox} holds.

Next, we show that $x\equiv 0$. First, we consider the case when there is an invertible  matrix $M$ such that  $M^{-1}AM=\begin{pmatrix}
\mu_1 & 0\\
0 & \mu_2
\end{pmatrix}$. Note here that, by Corollary~\ref{Lpqlinear2}, we know that $\Re(\mu_1)>0$ and $\Re(\mu_2)>0$. Letting $z= M^{-1}x$, then $z$ is  a solution of~\eqref{2Dlde2} and thus is of the form
\[ z(t) = \begin{pmatrix}
e^{\mu_1 t} & 0\\
0 & e^{\mu_2 t}
\end{pmatrix}
\begin{pmatrix}
c_1\\
c_2
\end{pmatrix}, \]
for some $c_1$, $c_2 \in \mathbb{C}$. Moreover,  $x(t) = Mz(t)$. Assume that $\begin{pmatrix}
c_1\\
c_2
\end{pmatrix}\ne \begin{pmatrix}
0\\
0
\end{pmatrix}$. Then
\begin{align*}
 &|x(t)-y(t)| = \left|M\begin{pmatrix}
c_1e^{\mu_1 t}\\
c_2e^{\mu_2 t}
\end{pmatrix}
+\frac{(q\gamma)^{\frac{1}{q}}\varepsilon e^{-\gamma t}}{|(\gamma A^{-1}+I_2)u|}A^{-1}u\right|\\
  &= e^{\min\{\Re(\mu_1),\Re(\mu_2)\}t}\left|M\begin{pmatrix}
c_1e^{\left(\mu_1-\min\{\Re(\mu_1),\Re(\mu_2)\}\right) t}\\
c_2e^{\left(\mu_2-\min\{\Re(\mu_1),\Re(\mu_2)\}\right) t}
\end{pmatrix}
+\frac{(q\gamma)^{\frac{1}{q}}\varepsilon e^{-\left(\gamma+\min\{\Re(\mu_1),\Re(\mu_2)\}\right) t}}{|(\gamma A^{-1}+I_2)u|}A^{-1}u\right|
\end{align*}
for $t \ge 0$. Since $\Re(\mu_1)>0$, $\Re(\mu_2)>0$ and $\begin{pmatrix}
c_1\\
c_2
\end{pmatrix}\ne \begin{pmatrix}
0\\
0
\end{pmatrix}$, we have that 
\[ \liminf_{t\to\infty}\left|M\begin{pmatrix}
c_1e^{\left(\mu_1-\min\{\Re(\mu_1),\Re(\mu_2)\}\right) t}\\
c_2e^{\left(\mu_2-\min\{\Re(\mu_1),\Re(\mu_2)\}\right) t}
\end{pmatrix}
+\frac{(q\gamma)^{\frac{1}{q}}\varepsilon e^{-\left(\gamma+\min\{\Re(\mu_1),\Re(\mu_2)\}\right) t}}{|(\gamma A^{-1}+I_2)u|}A^{-1}u\right| > 0, \]
and thus,
\[ \liminf_{t\to\infty}|x(t)-y(t)|=\infty. \]
This implies that $x-y \not\in L^p([0, \infty), \mathbb{C}^2)$, which yields a contradiction. We conclude that $\begin{pmatrix}
c_1\\
c_2
\end{pmatrix}= \begin{pmatrix}
0\\
0
\end{pmatrix}$, that is, $x \equiv 0$. Consequently, 
\[ |x(t)-y(t)| = \left|\frac{(q\gamma)^{\frac{1}{q}}\varepsilon e^{-\gamma t}}{|(\gamma A^{-1}+I_2)u|}A^{-1}u\right| = \frac{|A^{-1}u|(q\gamma)^{\frac{1}{q}}\varepsilon e^{-\gamma t}}{|(\gamma A^{-1}+I_2)u|}, \quad t \ge 0. \]
It follows  that
\[ \|x-y\|_{L^p([0, \infty), \mathbb{C}^2)} = \frac{|A^{-1}u|(q\gamma)^{\frac{1}{q}}}{|(\gamma A^{-1}+I_2)u|(p\gamma)^{\frac{1}{p}}}\varepsilon < \infty, \]
which gives that $L \ge \frac{|A^{-1}u|(q\gamma)^{\frac{1}{q}}}{|(\gamma A^{-1}+I_2)u|(p\gamma)^{\frac{1}{p}}}$.

The case when there is an invertible matrix $M$ such that $M^{-1}AM=\begin{pmatrix}
\nu & 1\\
0 & \nu
\end{pmatrix}$ can be treated using the same arguments as in the first case, and thus we omit the details. The proof is complete.
\end{proof}

\begin{example}\label{ex3}
Consider the autonomous linear system
\begin{equation}\label{2Dlde-ex3}
x'=\begin{pmatrix}
\mu_1 & 0\\
0 & \mu_2
\end{pmatrix}x, \quad t\ge 0,
\end{equation}
where $\mu_1>0$ and $\mu_2>0$. Clearly, \eqref{2Dlde-ex3} admits an exponential expansion on $[0,\infty)$. Let $p=q\ge1$. By Corollary~\ref{Lpqlinear2}, \eqref{2Dlde-ex3} admits the $(L^p, L^p)$ Hyers-Ulam stability with an HUS constant $L>0$. Since $r=1$, $M=I_2$, $\Re(\mu_1)=\mu_1$ and $\Re(\mu_2)=\mu_2$ hold, we have
\[ L=\frac{1}{\min\{\mu_1,\mu_2\}}. \]

We will now show that the constant $L$ is the minimal HUS constant for \eqref{2Dlde-ex3}. Assume the contrary, i.e. that there exists an HUS constant $L_*>0$ such that $0<L_*<L$. By Theorem~\ref{Lower:thm}, we obtain
\[ L_* \ge \frac{|A^{-1}u|}{|(\gamma A^{-1}+I_2)u|}, \]
where $u$ is an arbitrary unit vector and $\gamma$ is an arbitrary positive constant. 
Let $u = \begin{pmatrix}
u_{1}\\
u_{2}
\end{pmatrix}$. Since $A^{-1}$ is given by
\[ A^{-1} = \begin{pmatrix}
\frac{1}{\mu_1} & 0\\
0 & \frac{1}{\mu_2}
\end{pmatrix}, \]
we have that
\[ |A^{-1}u|_{\infty} = \max\left\{\frac{1}{\mu_1}|u_{1}|, \frac{1}{\mu_2}|u_{2}|\right\}, \quad |(\gamma A^{-1}+I_2)u|_{\infty} = \max\left\{\left(\frac{\gamma}{\mu_1}+1\right)|u_{1}|, \left(\frac{\gamma}{\mu_2}+1\right)|u_{2}|\right\}. \]
First we consider the case when $\mu_1\ge \mu_2>0$. Since $u$ is an arbitrary unit vector, if we choose $u$ as $u_1=0$ and $u_2=1$, we get
\[ L_* \ge \frac{|A^{-1}u|_{\infty}}{|(\gamma A^{-1}+I_2)u|_{\infty}} = \frac{1}{\gamma+\mu_2}. \] Letting $\gamma \to 0$, we obtain that 
\[
L_*\ge \frac{1}{\mu_2}=L,
\]
which yields a  contradiction. In the case of $0<\mu_1<\mu_2$, a contradiction can be derived in a similar way. Therefore, \eqref{2Dlde-ex3} has Hyers-Ulam stability with the minimal HUS constant $L=\frac{1}{\min\{\mu_1,\mu_2\}}$ when $p=q\ge1$. This example argues that our results are sharp. 
\end{example}

\begin{remark}
Let $A=\begin{pmatrix}
\mu_1 & 0\\
0 & \mu_2
\end{pmatrix}$. Then the minimal HUS constant $L$ in Example \ref{ex3} is $L=\frac{1}{\min\{\mu_1,\mu_2\}} =\|A^{-1}\|_{\infty}$. The minimal HUS constant of this type has been obtained for the general matrix $A$ in the study of the classical Hyers-Ulam stability ($p=q=\infty$). We refer to~\cite[Theorem 4.3]{AndOni}. For other results dealing with the  minimal HUS constants for linear differential and difference operators, the reader is referred to \cite{BaiBlaPop1,BaiBlaPop2,BaiPop1,BaiPop2,BrzPopRasXu,MiuMiyTak,TakTakMiuMiy}. 
\end{remark}

\section*{Acknowledgments}
D. D. was supported in part by  the University of Rijeka under the projects uniri-iskusni-prirod-23-98
3046.
M. O. was supported by the Japan Society for the Promotion of Science (JSPS) KAKENHI (grant number JP20K03668).

%%%%%%%%%%%%%%%%
% Bibliography %
%%%%%%%%%%%%%%%%

\end{document}